\documentclass{article} 
\usepackage{amsmath,amsthm,graphicx,bm,xcolor}   

\theoremstyle{plain}
\newtheorem{theorem}{Theorem}
\newtheorem{lemma}[theorem]{Lemma}
\newtheorem*{theorem3'}{Theorem 3$'$}
\newtheorem*{theorem3''}{Theorem 3$''$}
\theoremstyle{definition}

\allowdisplaybreaks

\makeatletter 
\renewcommand\@biblabel[1]{}
\makeatother
             		
\title{How strong can the Parrondo effect be?}
\author{S. N. Ethier\thanks{Department of Mathematics, University of Utah, 155 S. 1400 E., Salt Lake City, UT 84112, USA. e-mail: ethier@math.utah.edu.  Partially supported by a grant from the Simons Foundation (429675).}\; and Jiyeon Lee\thanks{Department of Statistics, Yeungnam University, 280 Daehak-Ro, Gyeongsan, Gyeongbuk 38541, South Korea.  e-mail: leejy@yu.ac.kr. Supported by the Basic Science Research Program through the National Research Foundation of Korea (NRF) funded by the Ministry of Education (NRF-2018R1D1A1B07042307).}}
\date{}
\begin{document}
\maketitle

\begin{abstract}
If the parameters of the original Parrondo games $A$ and $B$ are allowed to be arbitrary, subject to a fairness constraint, and if the two (fair) games $A$ and $B$ are played in an arbitrary periodic sequence, then the rate of profit can not only be positive, it can be arbitrarily close to 1 (i.e., 100\%).  
\end{abstract}

\section{Introduction}\label{intro}

The Parrondo effect appears when two fair coin-tossing games, $A$ and $B$, played in a random sequence or in some periodic sequence such as $ABB\,ABB\,ABB\,\cdots$, form a winning game.  Let us define a $p$-coin to be a coin with probability $p$ of heads.  In the original capital-dependent games of Parrondo (Harmer and Abbott, 1999), game $A$ uses a fair coin, while game $B$ uses two biased coins, a $p_0$-coin if capital is congruent to 0 (mod 3) and a $p_1$-coin otherwise, where
\begin{equation}\label{p0,p1}
p_0=\frac{1}{10}\quad\text{and}\quad p_1=\frac34.
\end{equation}
(These coins can be physically realized with dice; see Figure~\ref{P-dice}.)  The player wins one unit with heads and loses one unit with tails.  Both games are fair, but the random mixture, denoted by $\frac12 A+\frac12 B$ and interpreted as the game in which the toss of a fair coin determines whether game $A$ or game $B$ is played, has long-term cumulative profit per game played (hereafter, rate of profit) 
\begin{equation*}
\mu\big({\textstyle\frac12}A+{\textstyle\frac12}B\big)=\frac{18}{709}\approx0.0253879,
\end{equation*}
and the pattern $ABB$ has rate of profit
\begin{equation}\label{ABB}
\mu(ABB)=\frac{2416}{35601}\approx0.0678633.
\end{equation}
Dinis (2008) found that the pattern $ABABB$ has the highest rate of profit, namely
\begin{equation}\label{ABABB}
\mu(ABABB)=\frac{3613392}{47747645}\approx0.0756769.
\end{equation}

\begin{figure}[t]
\centering
\includegraphics[width=4in]{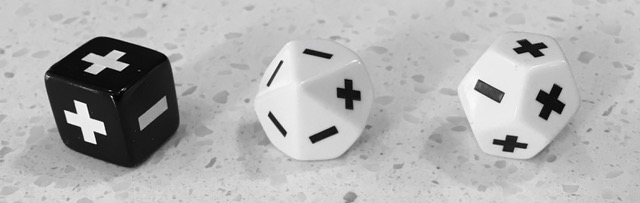}
\caption{\label{P-dice}Parrondo dice. Game $A$ uses the black die (3 $+$, 3 $-$), while game $B$ uses the white dice, namely (1 $+$, 9 $-$) when capital is congruent to 0 (mod~3) and (9 $+$, 3 $-$) otherwise.  The player wins one unit with a plus sign and loses one unit with a minus sign.  Both games are fair, but when alternated, either randomly or periodically, a winning game ensues.  (One exception: The pattern $AB$, that is, the periodic sequence $AB\,AB\,AB\,\cdots$, remains fair.)} 
\end{figure}

These rates of profit are rather modest.  Can we modify the games to make the rates of profit more substantial?  To put it more precisely, how large can the rate of profit be if we vary the parameters of the games, subject to a fairness constraint?  We will focus on periodic sequences, where the rates of profit tend to be larger than with random sequences.

Game $A$ is always the same fair-coin-tossing game.  With $r\ge3$ an integer, game $B$ is a mod $r$ capital-dependent game that uses two biased coins, a $p_0$-coin ($p_0<1/2$) if capital is congruent to 0 (mod $r$), and a $p_1$-coin ($p_1>1/2$) otherwise.  The probabilities $p_0$ and $p_1$ must be such that game $B$ is fair, which requires the constraint
$$
(1-p_0)(1-p_1)^{r-1}=p_0 p_1^{r-1},
$$
or equivalently,
\begin{equation}\label{rho-param}
p_0=\frac{\rho^{r-1}}{1+\rho^{r-1}}\quad\text{and}\quad p_1=\frac{1}{1+\rho}
\end{equation}
for some $\rho\in(0,1)$.  The special case of $r=3$ and $\rho=1/3$ gives \eqref{p0,p1}.  The games are played in some pattern $\Gamma(A,B)$, repeated ad infinitum.  We denote the rate of profit by $\mu(r,\rho,\Gamma(A,B))$, so that the rates of profit in \eqref{ABB} and \eqref{ABABB} in this notation become $\mu(3,1/3,ABB)$ and $\mu(3,1/3,ABABB)$.  

How large can $\mu(r,\rho,\Gamma(A,B))$ be?  The answer, perhaps surprisingly, is that it can be arbitrarily close to 1 (i.e., 100\%).

\begin{theorem}\label{sup=1}
\begin{equation*}
\sup_{r\ge3,\;\rho\in(0,1),\;\Gamma(A,B)\text{ arbitrary}}\mu(r,\rho,\Gamma(A,B))=1.
\end{equation*}
\end{theorem}
The proof is deferred to Section~\ref{rate section}. 

We can compute $\mu(r,\rho,\Gamma(A,B))$ for $r\ge3$ (the modulo number in game $B$) and pattern $\Gamma(A,B)$ as a function of $\rho$ (the parameter in \eqref{rho-param}).  Indeed, the method of Ethier and Lee (2009) applies if $r$ is odd, and generalizations of it apply if $r$ is even; see Section~\ref{SLLN section} for details.  For example,
\begin{equation}\label{mu(ABB) formula}
\mu(3,\rho,ABB)=\frac{(1 - \rho)^3 (1 + \rho) (1 + 2 \rho + \rho^2 + 2 \rho^3 + \rho^4)}{3 + 12 \rho + 
 20 \rho^2 + 28 \rho^3 + 36 \rho^4 + 28 \rho^5 + 20 \rho^6 + 12 \rho^7 + 3 \rho^8}.
\end{equation}
This and other examples suggest that typically $\mu(r,\rho,\Gamma(A,B))$ is decreasing in $\rho$, hence maximized at $\rho=0$.  (There are exceptions, which include, when $r\ge3$ is odd, $AB^s$ with $s\ge3$ odd.)  We excluded the case $\rho=0$ in \eqref{rho-param}, but now we want to include it.  We find that
\begin{equation}\label{3,0,ABB}
\mu(3,0,ABB)=\frac13
\end{equation}
(by \eqref{mu(ABB) formula}) and
\begin{equation}\label{3,0,ABABB}
\mu(3,0,ABABB)=\frac{9}{25}.
\end{equation}
Thus, we take $\rho=0$ in what follows.  

For a given $r\ge3$, we expect that we can maximize the rate of profit $\mu(r,0,\Gamma(A,B))$ with a pattern of the form
\begin{equation}\label{Gamma_{r,s}}
\Gamma(A,B)=(AB)^s B^{r-2}
\end{equation}
for some positive integer $s$.  Notice that this is $ABB$ if $(r,s)=(3,1)$ and $ABABB$ if $(r,s)=(3,2)$.  

Let us explain the intuition behind \eqref{Gamma_{r,s}}.  Only the $s$ plays of game $A$ are random.  Game $B$ is deterministic and very simple: If capital is congruent to 0 (mod $r$), we lose one unit, otherwise we win one unit.  Notice that cumulative profit remains bounded by $r$ when game $B$ is played repeatedly, hence cumulative profit per game played tends to 0 as the number of games played tends to infinity, and game $B$ is (asymptotically) fair.  

Clearly, the optimal strategy, if it were legal, would be to play game $A$ when capital is congruent to 0 (mod $r$) and to play game $B$ otherwise.  With initial capital congruent to 0 (mod $r$), this strategy could be described as playing the pattern $(AB)^S B^{r-2}$, where $S$ is the geometric random variable equal to the number of plays of game $A$ needed to achieve a win at that game.  Of course, random patterns are not ordinarily considered, so \eqref{Gamma_{r,s}} seems a reasonable nonrandom approximation for some positive integer $s$.

First, assume that $r$ is odd and initial capital is congruent to 0 (mod $r$).  If all $s$ plays of game $A$ result in losses, cumulative profit is $-1$ after one play of \eqref{Gamma_{r,s}}; otherwise it is $r$.  If initial capital is congruent to $r-1$ (mod $r$), then after one play of \eqref{Gamma_{r,s}}, cumulative profit is 1 with probability 1.   

Second, assume that $r$ is even and again initial capital is congruent to 0 (mod $r$).  If the number of wins in the $s$ plays of game $A$ is 0, cumulative profit is $0$ after one play of \eqref{Gamma_{r,s}}; if the number of wins is between 1 and $r/2$, inclusive, cumulative profit is $r$; if the number of wins is between $r/2+1$ and $r$, inclusive, cumulative profit is $2r$; if the number of wins is between $r+1$ and $3r/2$, inclusive, cumulative profit is $3r$; and so on.  If initial capital is congruent to $r-1$ (mod $r$), then after one play of \eqref{Gamma_{r,s}}, cumulative profit is 0 with probability 1.  

The probabilistic structure of capital growth after multiple plays of \eqref{Gamma_{r,s}} can be analyzed precisely from these observations, and we can evaluate the exact rate of profit.

\begin{theorem}\label{rate}
Let $r\ge3$ be an odd integer and $s$ be a positive integer.  Then
\begin{equation}\label{formula1}
\mu(r,0,(AB)^sB^{r-2})=\frac{r}{2s+r-2}\;\frac{2^s-1}{2^s+1},
\end{equation}
regardless of initial capital.

Let $r\ge4$ be an even integer and $s$ be a positive integer.  Then
\begin{equation}\label{formula2}
\mu(r,0,(AB)^sB^{r-2})=\begin{cases}\cfrac{r}{2s+r-2}\;\displaystyle{\sum_{k=0}^s\bigg\lceil\frac{2k}{r}\bigg\rceil\binom{s}{k}\frac{1}{2^s}}&\text{if initial capital is even},\\ \noalign{\medskip}
0&\text{if initial capital is odd}.\end{cases}
\end{equation}
\end{theorem}

The formula in \eqref{formula1} is consistent with \eqref{3,0,ABB} and \eqref{3,0,ABABB}.  The sum in \eqref{formula2} is equal to $(2^s-1)/2^s$ if $s\le r/2$ and bounded below by $(2^s-1)/2^s$ in general.  Theorem~\ref{rate} implies Theorem~\ref{sup=1}, as we will confirm later.  The proof of Theorem~\ref{rate} is deferred to Section~\ref{rate section}.  Table~\ref{rates} illustrates \eqref{formula1} with several examples.  

\begin{table}[htb]
\caption{\label{rates}The rate of profit $\mu(r,0,(AB)^s B^{r-2})$.  Here, for a given odd $r$, we choose $s$ to maximize $s'\mapsto\mu(r,0,(AB)^{s'} B^{r-2})$.  Results are rounded to six significant digits.}
\catcode`@=\active \def@{\hphantom{0}}
\catcode`#=\active \def#{\hphantom{$\,^1$}}
\tabcolsep=.2cm
\begin{center}
\begin{tabular}{ccrcccc}
$r$ & $s$ & $\mu(r,0,(AB)^s B^{r-2})$ &@@@& @@$r$ & @$s$ & $\mu(r,0,(AB)^s B^{r-2})$ \\
\noalign{\smallskip}
\cline{1-3}\cline{5-7}
\noalign{\smallskip}
3 & 2 &  $9/25=0.360000$@ && @@25 & @5 & 0.711662 \\
5 & 3 &  $35/81\approx0.432099$@ && @125 & @7 & 0.898263 \\
7 & 3 &  $49/99\approx0.494949$@ && @625 & @9 & 0.971238 \\
9 & 3 &   $7/13\approx0.538462$@ && 3125 & 11 & 0.992671 \\
\noalign{\smallskip}
\cline{1-3}\cline{5-7}
\end{tabular}
\end{center}
\end{table}

We do not consider random mixtures $\gamma A+(1-\gamma)B$ of games $A$ and $B$.  Although we expect that the rate of profit, which we denote by $\mu(r,\rho,\gamma A+(1-\gamma)B)$, can be made arbitrarily close to 1 by suitable choice of the modulo number $r$ in game $B$, the parameter $\rho$ in \eqref{rho-param}, and the probability $\gamma$ with which game $A$ is played, we cannot prove it.  However, see Table~\ref{rates-mixture} for several examples.

\begin{table}[htb]
\caption{\label{rates-mixture}The rate of profit $\mu(r,0,\gamma A+(1-\gamma)B)$.  Here, for a given odd $r$, we choose $\gamma$ to maximize $\gamma'\mapsto\mu(r,0,\gamma' A+(1-\gamma')B)$. Results are rounded to six significant digits.}
\catcode`@=\active \def@{\hphantom{0}}
\catcode`#=\active \def#{\hphantom{$\,^1$}}
\tabcolsep=.2cm
\begin{center}
\begin{tabular}{cccccccc}
    &          & $\mu(r,0,\gamma A$ &@@@&      &          & $\mu(r,0,\gamma A$ \\
$r$ & $\gamma$ & ${}+(1-\gamma)B)$    &@@@& @$r$ & $\gamma$ & ${}+(1-\gamma)B)$    \\
\noalign{\smallskip}
\cline{1-3}\cline{5-7}
\noalign{\smallskip}
3 & 0.407641 & 0.133369 && @@25 & 0.277926@ & 0.482769 \\
5 & 0.420756 & 0.229111 && @125 & 0.150722@ & 0.709914 \\
7 & 0.399201 & 0.279864 && @625 & 0.0739646 & 0.854806 \\
9 & 0.376138 & 0.318393 && 3125 & 0.0345306 & 0.931535 \\
\noalign{\smallskip}
\cline{1-3}\cline{5-7}
\end{tabular}
\end{center}
\end{table}

\section{SLLN for periodic sequences of games}\label{SLLN section}

Ethier and Lee (2009) proved a strong law of large numbers and a central limit theorem for periodic sequences of Parrondo games of the form $A^r B^s$, repeated ad infinitum, where $r$ and $s$ are positive integers.  Below we state a generalization of the SLLN to arbitrary patterns.  Later we will weaken the hypotheses as needed.

First, it should be mentioned that several other authors have studied periodic sequences of Parrondo games.  Pyke (2003) discussed one example, $AABB$, which he regarded as the alternation of $AA$ and $BB$.  His method is sound but his stated ``asymptotic average gain'' for that example is inaccurate, and the source of the error is unknown.  Kay and Johnson (2003) studied patterns of the form $A^r B^s$ in the context of history-dependent Parrondo games, and gave an expression for the rate of profit that is consistent with \eqref{mean-formula} below.  Key, K\l osek, and Abbott (2006), as well as R\'emillard and Vaillancourt (2019), took a different approach, analyzing periodic sequences of Parrondo games in terms of transience to $\pm\infty$ and recurrence instead of in terms of the rate of profit.

\begin{theorem}\label{SLLN}
Let $\bm P_A$ and $\bm P_B$ be transition matrices for Markov chains in a finite state space $\Sigma$.  Let $C_1C_2\cdots C_t$, where each $C_i$ is $A$ or $B$, be a pattern of $A$s and $B$s of length $t$.  Assume that $\bm P:=\bm P_{C_1}\bm P_{C_2}\cdots\bm P_{C_t}$ is irreducible and aperiodic, and let the row vector $\bm\pi$ be the unique stationary distribution of $\bm P$.  Given a real-valued function $w$ on $\Sigma\times\Sigma$,  define the payoff matrix $\bm W:=(w(i,j))_{i,j\in\Sigma}$.  Define $\dot{\bm P}_A:=\bm P_A\circ\bm W$ and $\dot{\bm P}_B:=\bm P_B\circ\bm W$, where $\circ$ denotes the Hadamard $($entrywise$)$ product, and put
\begin{equation}\label{mean-formula}
\mu:=t^{-1}\bm\pi(\dot{\bm P}_{C_1}+\bm P_{C_1}\dot{\bm P}_{C_2}+\cdots+\bm P_{C_1}\bm P_{C_2}\cdots\bm P_{C_{t-1}}\dot{\bm P}_{C_t})\bm1,
\end{equation}
where $\bm1$ denotes a column vector of $1$s with entries indexed by $\Sigma$.  Let $\{X_n\}_{n\ge0}$ be a nonhomogeneous Markov chain in $\Sigma$ with transition matrices $\bm P_{C_1}$, $\bm P_{C_2}$, \dots, $\bm P_{C_t}$, $\bm P_{C_1}$, $\bm P_{C_2}$, \dots, $\bm P_{C_t}$, $\bm P_{C_1}$, and so on, and let the initial distribution be arbitrary.  For each $n\ge1$, define $\xi_n:=w(X_{n-1},X_n)$ and $S_n:=\xi_1+\cdots+\xi_n$.  Then $\lim_{n\to\infty}n^{-1}S_n=\mu$ a.s.
\end{theorem}

\begin{proof}
The proof is identical to the proof of Theorem~6 of Ethier and Lee (2009).  However, here we have assumed fewer hypotheses and should explain why.  First, it is unnecessary to assume that $\bm P_A$ and $\bm P_B$ are irreducible and aperiodic because that assumption is not needed.  It is also unnecessary to assume that all cyclic permutations of $\bm P:=\bm P_{C_1}\bm P_{C_2}\cdots\bm P_{C_t}$ are irreducible and aperiodic because that assumption is redundant; it suffices that $\bm P$ itself be irreducible and aperiodic.  Finally, we assumed in the original theorem that the Markov chain 
\begin{equation}\label{statMC2}
(X_0,X_1,\ldots,X_t),(X_t,X_{t+1},\ldots,X_{2t}),(X_{2t},X_{2t+1},\ldots,X_{3t}),\ldots
\end{equation}
is irreducible and aperiodic, and we claim that this assumption is also redundant.  The state space $\Sigma^*$ of \eqref{statMC2} is the set of $(x_0,x_1,\ldots,x_t)\in\Sigma^{t+1}$ such that 
$$
\bm\pi(x_0)\bm P_{C_1}(x_0,x_1)\bm P_{C_2}(x_1,x_2)\cdots\bm P_{C_t}(x_{t-1},x_t)>0,
$$
and its transition matrix $\bm Q$ is given by
\begin{align*}
&\bm Q((x_0,x_1,\ldots,x_t),(x_t,x_{t+1},\ldots,x_{2t}))\\
&\quad{}=\bm P_{C_1}(x_t,x_{t+1})\bm P_{C_2}(x_{t+1},x_{t+2})\cdots\bm P_{C_t}(x_{2t-1},x_{2t}).
\end{align*}
We use the fact that a necessary and sufficient condition for a finite Markov chain to be irreducible and aperiodic is that some power of its transition matrix has all entries positive.  It is straightforward to show that $\bm Q^n$ has all entries positive if $\bm P^{n-1}$ does.  Indeed, 
\begin{align}\label{Q^n}
&\bm Q^n((x_0,x_1,\ldots,x_t),(y_0,y_1,\ldots,y_t))\nonumber\\
&\quad{}=\bm P^{n-1}(x_t,y_0)\bm P_{C_1}(y_0,y_1)\bm P_{C_2}(y_1,y_2)\cdots\bm P_{C_t}(y_{t-1},y_t).
\end{align}
Because $\bm P$ is irreducible and aperiodic, so too is $\bm Q$.
\end{proof}

As an illustration, we can use \eqref{mean-formula} to confirm \eqref{ABB} and \eqref{ABABB}, in which case $\Sigma=\{0,1,2\}$,
$$
\bm P_A=\begin{pmatrix}0&1/2&1/2\\1/2&0&1/2\\1/2&1/2&0\end{pmatrix},\quad
\bm P_B=\begin{pmatrix}0&1/10&9/10\\1/4&0&3/4\\3/4&1/4&0\end{pmatrix},
$$
and the payoff matrix is
$$
\bm W=\begin{pmatrix}0&1&-1\\-1&0&1\\ 1&-1&0\end{pmatrix}.
$$

More generally, we wish to apply Theorem~\ref{SLLN} with 
\begin{equation}\label{Sigma}
\Sigma=\{0,1,\ldots,r-1\}
\end{equation}
($r$ is the modulo number in game $B$), the $r\times r$ transition matrices
\begin{equation}\label{PA}
{\bm P}_A=\begin{pmatrix}0&1/2&0&\cdots&0&0&1/2\\
1/2&0&1/2&\cdots&0&0&0\\
0&1/2&0&\cdots&0&0&0\\
\vdots&\vdots&\vdots& &\vdots&\vdots&\vdots\\
0&0&0&\cdots&0&1/2&0\\
0&0&0&\cdots&1/2&0&1/2\\
1/2&0&0&\cdots&0&1/2&0\end{pmatrix},
\end{equation}
\begin{equation}\label{PB}
{\bm P}_B=\begin{pmatrix}0&p_0&0&\cdots&0&0&1-p_0\\
1-p_1&0&p_1&\cdots&0&0&0\\
0&1-p_1&0&\cdots&0&0&0\\
\vdots&\vdots&\vdots& &\vdots&\vdots&\vdots\\
0&0&0&\cdots&0&p_1&0\\
0&0&0&\cdots&1-p_1&0&p_1\\
p_1&0&0&\cdots&0&1-p_1&0\end{pmatrix},
\end{equation}
where $p_0$ and $p_1$ are given by \eqref{rho-param}, and the $r\times r$ payoff matrix
\begin{equation}\label{W}
\bm W=\begin{pmatrix}0&1&0&\cdots&0&0&-1\\
-1&0&1&\cdots&0&0&0\\
0&-1&0&\cdots&0&0&0\\
\vdots&\vdots&\vdots& &\vdots&\vdots&\vdots\\
0&0&0&\cdots&0&1&0\\
0&0&0&\cdots&-1&0&1\\
1&0&0&\cdots&0&-1&0\end{pmatrix}.
\end{equation}
There are five cases that we want to consider.

\begin{enumerate}
\item Let the pattern $C_1C_2\cdots C_t$ of Theorem~\ref{SLLN} be arbitrary.  If $\rho>0$ and $r$ is odd ($\ge3$), then $\bm P:=\bm P_{C_1}\bm P_{C_2}\cdots\bm P_{C_t}$ is irreducible and aperiodic.

\item Let the pattern $C_1C_2\cdots C_t$ be arbitrary.  If $\rho>0$, $r$ is even ($\ge4$), and $t$ is odd, then $\bm P$ is irreducible and periodic with period 2.  

\item Let the pattern $C_1C_2\cdots C_t$ be arbitrary.  If $\rho>0$, $r$ is even ($\ge4$), and $t$ is even, then $\bm P$ is reducible with two aperiodic recurrent classes, each of size $r/2$.  

\item Let the pattern $C_1C_2\cdots C_t$ have the form $(AB)^s B^{r-2}$ for a positive integer $s$.  If $\rho=0$ and $r$ is odd ($\ge3$), then $\bm P:=(\bm P_A \bm P_B)^s(\bm P_B)^{r-2}$ is reducible with one aperiodic recurrent class of size 2 and $r-2$ transient states.

\item Let the pattern $C_1C_2\cdots C_t$ have the form $(AB)^s B^{r-2}$ for a positive integer $s$.  If $\rho=0$ and $r$ is even ($\ge4$), then $\bm P$ is reducible with two absorbing states and $r-2$ transient states.
\end{enumerate}

Theorem~\ref{SLLN} applies directly only to Case 1.  Nevertheless, the theorem can be extended so as to apply first to Cases 1 and 4, then to Cases 3 and 5, and finally to Case 2.  We begin by generalizing Theorem \ref{SLLN} so as to apply to Cases 1 and 4.

\begin{theorem3'}
Theorem~\ref{SLLN} holds with ``is irreducible and aperiodic'' replaced by ``has only one recurrent class, which is aperiodic''.
\end{theorem3'}

\begin{proof}
Assume that $\bm P$ has only one recurrent class, which is aperiodic.  Let $\Sigma_0\subset\Sigma$ be the unique recurrent class.  The stationary distribution $\bm\pi$ of $\bm P$ is unique and satisfies $\bm\pi(x)>0$ if $x\in\Sigma_0$ and $\bm\pi(x)=0$ otherwise.  For some $n\ge2$, $\bm P^{n-1}(x_t,y_0)>0$ for all $x_t,y_0\in\Sigma_0$.  With the help of \eqref{Q^n} we find that $\bm Q^n$ has all entries positive, hence $\bm Q$ is irreducible and aperiodic.  

An example may help to clarify this argument.  Consider the special case of \eqref{Sigma}--\eqref{W} (with \eqref{rho-param}) in which $\rho=0$ and $r=3$, and let $C_1C_2C_3=ABB$.  Then $\bm\pi=(2/3,0,1/3)$, and the state space for the Markov chain $(X_0,X_1,X_2,X_3)$, $(X_3,X_4,X_5,X_6)$, \dots\ is $\Sigma^*=\{(0,1,2,0), (0,2,0,2), (2,0,2,0), (2,1,2,0)\}$ with corresponding transition matrix
$$
\bm Q=\begin{pmatrix}1/2&1/2&0&0\\0&0&1/2&1/2\\1/2&1/2&0&0\\1/2&1/2&0&0\end{pmatrix},
$$
which is irreducible and aperiodic.

The remainder of the proof follows that of Theorem~6 of Ethier and Lee (2009).
\end{proof}

We turn to Cases 3 and 5, which require a new formulation of Theorem \ref{SLLN}, the difficulty being that the limit in the SLLN depends on the initial distribution of the underlying Markov chain.

\begin{theorem}\label{SLLN-2}
Let $\bm P_A$ and $\bm P_B$ be transition matrices for Markov chains in a finite state space $\Sigma$.  Let $C_1C_2\cdots C_t$, where each $C_i$ is $A$ or $B$, be a pattern of $A$s and $B$s of length $t$.  Assume that $\bm P:=\bm P_{C_1}\bm P_{C_2}\cdots\bm P_{C_t}$ is reducible with two recurrent classes $R_1$ and $R_2$, both of which are aperiodic, and possibly some transient states, and let the row vectors $\bm\pi_1$ and $\bm\pi_2$ be the unique stationary distributions of $\bm P$ concentrated on $R_1$ and $R_2$, respectively.  Given a real-valued function $w$ on $\Sigma\times\Sigma$,  define the payoff matrix $\bm W:=(w(i,j))_{i,j\in\Sigma}$.  Define $\dot{\bm P}_A:=\bm P_A\circ\bm W$ and $\dot{\bm P}_B:=\bm P_B\circ\bm W$, where $\circ$ denotes the Hadamard $($entrywise$)$ product, and put
\begin{equation*}
\mu_j:=t^{-1}\bm\pi_j(\dot{\bm P}_{C_1}+\bm P_{C_1}\dot{\bm P}_{C_2}+\cdots+\bm P_{C_1}\bm P_{C_2}\cdots\bm P_{C_{t-1}}\dot{\bm P}_{C_t})\bm1
\end{equation*}
for $j=1,2$, where $\bm1$ denotes a column vector of $1$s with entries indexed by $\Sigma$.  Let $\{X_n\}_{n\ge0}$ be a nonhomogeneous Markov chain in $\Sigma$ with transition matrices $\bm P_{C_1}$, $\bm P_{C_2}$, \dots, $\bm P_{C_t}$, $\bm P_{C_1}$, $\bm P_{C_2}$, \dots, $\bm P_{C_t}$, $\bm P_{C_1}$, and so on, and let its initial state be $i_0\in\Sigma$.  Let $\alpha:=P(X_{nt}\in R_1\text{ for }n\text{ sufficiently large})$.  For each $n\ge1$, define $\xi_n:=w(X_{n-1},X_n)$ and $S_n:=\xi_1+\cdots+\xi_n$.  Then $\lim_{n\to\infty}n^{-1}S_n=\alpha\mu_1+(1-\alpha)\mu_2$ a.s.
\end{theorem}

\begin{proof}
The argument used to prove the conclusion of Theorem \ref{SLLN} when $\bm\pi$ is the initial distribution applies here, allowing us to prove that $\lim_{n\to\infty}n^{-1}S_n=\mu_j$ a.s.\ if $\bm\pi_j$ is the initial distribution, then if the initial state $i_0$ belongs to $R_j$, for $j=1,2$.  Let $N:=\min\{nt: X_{nt}\in R_1\cup R_2\}$.  Then $P(X_N\in R_1)=\alpha$, and the stated conclusion readily follows.
\end{proof}

We conclude this section by addressing Case 2.

\begin{theorem3''}
Theorem~\ref{SLLN} holds with ``is irreducible and aperiodic'' replaced by ``is irreducible and periodic with period 2''.
\end{theorem3''}

\begin{proof}
The idea is to apply Theorem \ref{SLLN-2} with the pattern $C_1C_2\cdots C_t$ replaced by the pattern $C_1C_2\cdots C_tC_1C_2\cdots C_t$, which has the same limit in the SLLN.  In particular, $\bm P$ is replaced by $\bm P^2$.  The assumption that $\bm P$ is irreducible with period 2 means that $\Sigma$ is the disjoint union of $R_1$ and $R_2$, and transitions under $\bm P$ take $R_1$ to $R_2$ and $R_2$ to $R_1$.  This means that $\bm P^2$ is reducible with two recurrent classes, $R_1$ and $R_2$, and no transient states.  Let the row vectors $\bm\pi_1$ and $\bm\pi_2$ be the unique stationary distributions of $\bm P^2$ concentrated on $R_1$ and $R_2$, respectively.  Then $\bm\pi_1\bm P=\bm\pi_2$ and $\bm\pi_2\bm P=\bm\pi_1$.  Consequently, the limit $\mu_1$ starting in $R_1$ is, according to Theorem \ref{SLLN-2}, 
\begin{align*}
&(2t)^{-1}\bm\pi_1\big[\dot{\bm P}_{C_1}+\bm P_{C_1}\dot{\bm P}_{C_2}+\cdots+\bm P_{C_1}\bm P_{C_2}\cdots\bm P_{C_{t-1}}\dot{\bm P}_{C_t}\\
&\qquad\qquad\;{}+\bm P(\dot{\bm P}_{C_1}+\bm P_{C_1}\dot{\bm P}_{C_2}+\cdots+\bm P_{C_1}\bm P_{C_2}\cdots\bm P_{C_{t-1}}\dot{\bm P}_{C_t})\big]\bm1\\
&\quad{}=t^{-1}\bm\pi(\dot{\bm P}_{C_1}+\bm P_{C_1}\dot{\bm P}_{C_2}+\cdots+\bm P_{C_1}\bm P_{C_2}\cdots\bm P_{C_{t-1}}\dot{\bm P}_{C_t})\bm1,
\end{align*}
where $\bm\pi:=(\bm\pi_1+\bm\pi_2)/2$ is the unique stationary distribution of $\bm P$, and this is \eqref{mean-formula}.  The limit $\mu_2$ starting in $R_2$ is the same but with $\bm\pi_1$ and $\bm\pi_2$ interchanged, and again this is \eqref{mean-formula}.
\end{proof}

For example, we find that
$$
\mu(4,\rho,ABB)=\frac{(1 - \rho)^3}{3 (1 + \rho^3)}
$$
as a consequence of Theorem $3''$, and
$$
\mu(4,\rho,ABBB)=\begin{cases}\cfrac{(1 - \rho)(2 - 3 \rho + 2 \rho^2)}{4 (1 + \rho) (1 - \rho + \rho^2)}&\text{if initial capital is even}\\ \noalign{\medskip}
-\cfrac{\rho^2(1 - \rho)(5-6\rho+5\rho^2)}{4(1 + \rho)^3 (1 - \rho + \rho^2)^2}&\text{if initial capital is odd}\end{cases}
$$
as a consequence of Theorem \ref{SLLN-2}.  Recalling the five cases below \eqref{Sigma}--\eqref{W}, these two examples correspond to Cases 2 and 3, respectively, whereas \eqref{mu(ABB) formula} corresponds to Case 1.

Finally, we point out that R\'emillard and Vaillancourt (2019) have addressed some of the same issues that we encountered in this section, namely reducibility, periodicity, and more than one recurrent class, albeit by different methods.

\section{Mean of a binomial-like distribution}\label{mean section}

Here we want to find the mean of a discrete distribution that depends, like the binomial, on two parameters, a positive integer $n$ and $p\in(0,1)$.  The distribution does not appear to have a name.  The formula for the probability mass function depends on whether $n$ is even or odd, so we treat the two cases separately.  We use the convention that $q:=1-p$. 

In the case $n=2m$ with $m$ a positive integer, consider a particle that starts at $(0,0)$.  At each time step, it moves one unit to the right with probability $p$ or one unit up with probability $q$, stopping at the first time it reaches the boundary $(k,m-\lfloor k/2\rfloor)$, $k=0,1,\ldots,2m$.  Let $Z_{2m}$ denote the $x$-coordinate of its final position.  Then
\begin{equation}\label{dist-even}
P(Z_{2m}=k)=\binom{m + \lfloor k/2\rfloor}{k}p^k q^{m - \lfloor k/2\rfloor},\quad k=0,1,\ldots,2m.
\end{equation}
Each lattice path ending at $(k,m - \lfloor k/2\rfloor)$ has probability $p^k q^{m - \lfloor k/2\rfloor}$, and
the binomial coefficient counts the number of paths that end at $(k,m-k/2)$ if $k$ is even, and at $(k,m-(k-1)/2)$ if $k$ is odd because in the latter case the path must first reach $(k,m-(k+1)/2)$.  See Figure~\ref{even-fig}.

In the case $n=2m-1$ with $m$ a positive integer, again consider a particle that starts at $(0,0)$.  At each time step, it moves one unit to the right with probability $p$ or one unit up with probability $q$, stopping at the first time it reaches the boundary $(k,m-\lceil k/2\rceil)$, $k=0,1,\ldots,2m-1$.  Let $Z_{2m-1}$ denote the $x$-coordinate of its final position.  Then
\begin{equation}\label{dist-odd}
P(Z_{2m-1}=k)=\binom{m - 1 + \lceil k/2\rceil}{k}p^k q^{m - \lceil k/2\rceil},\quad k=0,1,\ldots,2m-1.
\end{equation}
Each lattice path ending at $(k,m - \lceil k/2\rceil)$ has probability $p^k q^{m - \lceil k/2\rceil}$, and
the binomial coefficient counts the number of paths that end at $(k,m-(k+1)/2)$ if $k$ is odd, and at $(k,m-k/2)$ if $k$ is even because in the latter case the path must first reach $(k,m-1-k/2)$.  See Figure~\ref{odd-fig}.

\begin{figure}
\setlength{\unitlength}{1cm}
\begin{picture}(10.5,6.5)
\thinlines
\put(1,1){\line(1,0){9.5}}
\put(1,2){\line(1,0){9.5}}
\put(1,3){\line(1,0){7.5}}
\put(1,4){\line(1,0){5.5}}
\put(1,5){\line(1,0){3.5}}
\put(1,6){\line(1,0){1.5}}
\put(1,1){\line(0,1){5.5}}
\put(2,1){\line(0,1){5.5}}
\put(3,1){\line(0,1){4.5}}
\put(4,1){\line(0,1){4.5}}
\put(5,1){\line(0,1){3.5}}
\put(6,1){\line(0,1){3.5}}
\put(7,1){\line(0,1){2.5}}
\put(8,1){\line(0,1){2.5}}
\put(9,1){\line(0,1){1.5}}
\put(10,1){\line(0,1){1.5}}
\put(1,4){\circle*{0.2}}
\put(2,4){\circle*{0.2}}
\put(3,3){\circle*{0.2}}
\put(4,3){\circle*{0.2}}
\put(5,2){\circle*{0.2}}
\put(6,2){\circle*{0.2}}
\put(7,1){\circle*{0.2}}
\put(1,5){\circle{0.2}}
\put(2,5){\circle{0.2}}
\put(3,4){\circle{0.2}}
\put(4,4){\circle{0.2}}
\put(5,3){\circle{0.2}}
\put(6,3){\circle{0.2}}
\put(7,2){\circle{0.2}}
\put(8,2){\circle{0.2}}
\put(9,1){\circle{0.2}}
\put(0.9,0.5){0}
\put(1.9,0.5){1}
\put(2.9,0.5){2}
\put(3.9,0.5){3}
\put(4.9,0.5){4}
\put(5.9,0.5){5}
\put(6.9,0.5){6}
\put(7.9,0.5){7}
\put(8.9,0.5){8}
\put(9.9,0.5){9}
\put(0.5,0.9){0}
\put(0.5,1.9){1}
\put(0.5,2.9){2}
\put(0.5,3.9){3}
\put(0.5,4.9){4}
\put(0.5,5.9){5}
\end{picture}
\caption{\label{even-fig}The solid dots determine the boundary characterizing $Z_6$, whereas the open dots determine the boundary characterizing $Z_8$.}
\end{figure}
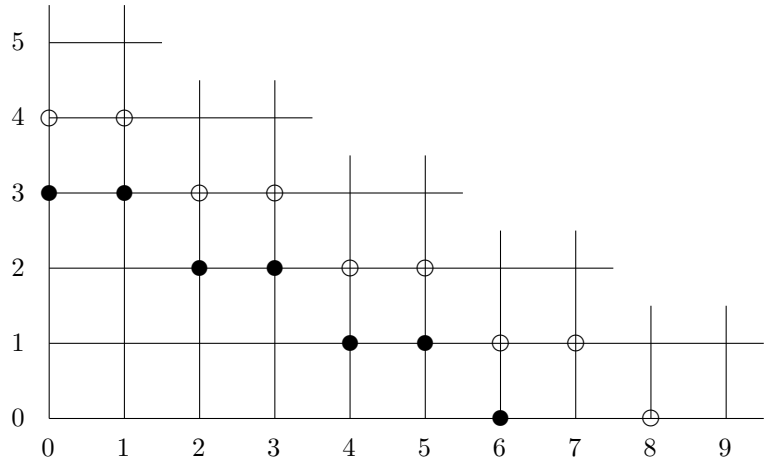

\begin{figure}
\setlength{\unitlength}{1cm}
\begin{picture}(10.5,6.5)
\thinlines
\put(1,1){\line(1,0){9.5}}
\put(1,2){\line(1,0){9.5}}
\put(1,3){\line(1,0){7.5}}
\put(1,4){\line(1,0){5.5}}
\put(1,5){\line(1,0){3.5}}
\put(1,6){\line(1,0){1.5}}
\put(1,1){\line(0,1){5.5}}
\put(2,1){\line(0,1){5.5}}
\put(3,1){\line(0,1){4.5}}
\put(4,1){\line(0,1){4.5}}
\put(5,1){\line(0,1){3.5}}
\put(6,1){\line(0,1){3.5}}
\put(7,1){\line(0,1){2.5}}
\put(8,1){\line(0,1){2.5}}
\put(9,1){\line(0,1){1.5}}
\put(10,1){\line(0,1){1.5}}
\put(1,4){\circle*{0.2}}
\put(2,3){\circle*{0.2}}
\put(3,3){\circle*{0.2}}
\put(4,2){\circle*{0.2}}
\put(5,2){\circle*{0.2}}
\put(6,1){\circle*{0.2}}
\put(1,5){\circle{0.2}}
\put(2,4){\circle{0.2}}
\put(3,4){\circle{0.2}}
\put(4,3){\circle{0.2}}
\put(5,3){\circle{0.2}}
\put(6,2){\circle{0.2}}
\put(7,2){\circle{0.2}}
\put(8,1){\circle{0.2}}
\put(0.9,0.5){0}
\put(1.9,0.5){1}
\put(2.9,0.5){2}
\put(3.9,0.5){3}
\put(4.9,0.5){4}
\put(5.9,0.5){5}
\put(6.9,0.5){6}
\put(7.9,0.5){7}
\put(8.9,0.5){8}
\put(9.9,0.5){9}
\put(0.5,0.9){0}
\put(0.5,1.9){1}
\put(0.5,2.9){2}
\put(0.5,3.9){3}
\put(0.5,4.9){4}
\put(0.5,5.9){5}
\end{picture}
\caption{\label{odd-fig}The solid dots determine the boundary characterizing $Z_5$, whereas the open dots determine the boundary characterizing $Z_7$.}
\end{figure}
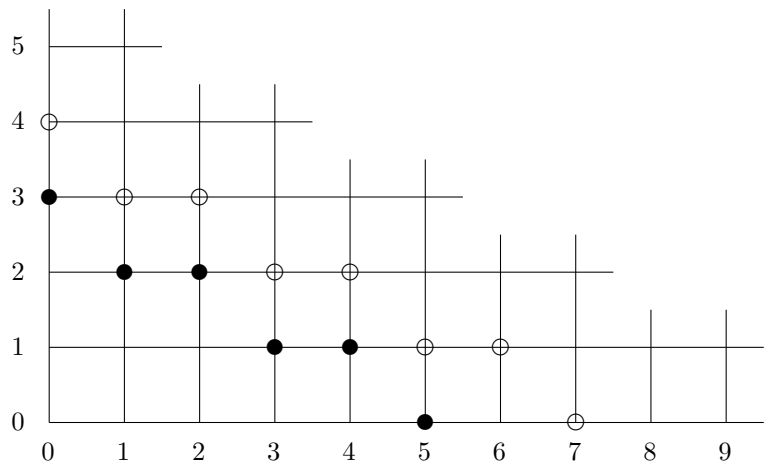

\begin{lemma}\label{even-lemma}
\begin{equation}\label{even}
P(Z_n\text{ is even})=\begin{cases}(1+q^{n+1})/(1+q)&\text{if $n$ is even},\\
(q+q^{n+1})/(1+q)&\text{if $n$ is odd}.\end{cases}
\end{equation}
Equivalently,
\begin{equation*}
P(Z_n\text{ is odd})=\begin{cases}(q-q^{n+1})/(1+q)&\text{if $n$ is even},\\
(1-q^{n+1})/(1+q)&\text{if $n$ is odd}.\end{cases}
\end{equation*}
\end{lemma}

\begin{proof}
We give separate proofs for $n$ even and $n$ odd, both by induction.  To initialize, in the $n=1$ case, the probability mass function is $q$ at 0 and $p$ at 1, so \eqref{even} holds.  In the $n=2$ case, the probability mass function is $q$ at 0, $pq$ at 1, and $p^2$ at 2, so again \eqref{even} holds.

Now assume that \eqref{even} holds for $n=2m$.  We must show that it holds for $n=2m+2$.  By the interpretation of the distribution (see Figure~\ref{even-fig}),
\begin{align*}
P(Z_{2m+2}\text{ is even}\mid Z_{2m}\text{ is even})&=q+p^2=1-q+q^2,\\
P(Z_{2m+2}\text{ is even}\mid Z_{2m}\text{ is odd})&=p=1-q.
\end{align*}
We conclude that
\begin{align*}
P(Z_{2m+2}\text{ is even})&=P(Z_{2m}\text{ is even})P(Z_{2m+2}\text{ is even}\mid Z_{2m}\text{ is even})\\
&\qquad{}+P(Z_{2m}\text{ is odd})P(Z_{2m+2}\text{ is even}\mid Z_{2m}\text{ is odd})\\
&=\frac{1+q^{2m+1}}{1+q}\,(1-q+q^2)+\frac{q-q^{2m+1}}{1+q}\,(1-q)\\
&=\frac{1+q^{2m+3}}{1+q},
\end{align*}
proving the lemma when $n$ is even.

Now assume that \eqref{even} holds for $n=2m-1$.  We must show that it holds for $n=2m+1$.  By the interpretation of the distribution (see Figure~\ref{odd-fig}),
\begin{align*}
P(Z_{2m+1}\text{ is even}\mid Z_{2m-1}\text{ is even})&=q,\\
P(Z_{2m+1}\text{ is even}\mid Z_{2m-1}\text{ is odd})&=pq=q(1-q).
\end{align*}
We conclude that
\begin{align*}
P(Z_{2m+1}\text{ is even})&=P(Z_{2m-1}\text{ is even})P(Z_{2m+1}\text{ is even}\mid Z_{2m-1}\text{ is even})\\
&\qquad{}+P(Z_{2m-1}\text{ is odd})P(Z_{2m+1}\text{ is even}\mid Z_{2m-1}\text{ is odd})\\
&=\frac{q+q^{2m}}{1+q}\,q+\frac{1-q^{2m}}{1+q}\,q(1-q)\\
&=\frac{q+q^{2m+2}}{1+q},
\end{align*}
proving the lemma when $n$ is odd.
\end{proof}

\begin{lemma}\label{mean-lemma}
\begin{equation*}
E[Z_n]=n\,\frac{p}{2-p}+[1-(-1)^n(1-p)^n]\,\frac{p(1-p)}{(2-p)^2}.
\end{equation*}
Equivalently,
\begin{equation}\label{mean}
E[Z_n]=n\,\frac{1-q}{1+q}+[1-(-1)^n q^n]\,\frac{q(1-q)}{(1+q)^2}.
\end{equation}
\end{lemma}

\begin{proof}
As with Lemma~\ref{even-lemma}, we give separate proofs for $n$ even and $n$ odd, both by induction.  To initialize, in the $n=1$ case, the probability mass function is $q$ at 0 and $p$ at 1, so the mean is $p=1-q$ and \eqref{mean} holds.  In the $n=2$ case, the probability mass function is $q$ at 0, $pq$ at 1, and $p^2$ at 2, so the mean is $pq+2p^2=(1-q)(2-q)$ and again \eqref{mean} holds.

Now assume that \eqref{mean} holds for $n=2m$.  We must show that it holds for $n=2m+2$.  By the interpretation of the distribution (see Figure~\ref{even-fig}),
\begin{align*}
E[Z_{2m+2}-Z_{2m}\mid Z_{2m}\text{ is even}]&=pq+2p^2=(1-q)(2-q),\\
E[Z_{2m+2}-Z_{2m}\mid Z_{2m}\text{ is odd}]&=p=1-q.
\end{align*}
We conclude from the induction hypothesis and Lemma~\ref{even-lemma} that
\begin{align*}
E[Z_{2m+2}]&=E[Z_{2m}]+P(Z_{2m}\text{ is even})E[Z_{2m+2}-Z_{2m}\mid Z_{2m}\text{ is even}]\\
&\qquad\qquad\quad{}+P(Z_{2m}\text{ is odd})E[Z_{2m+2}-Z_{2m}\mid Z_{2m}\text{ is odd}]\\
&=2m\,\frac{1-q}{1+q}+(1-q^{2m})\,\frac{q(1-q)}{(1+q)^2}\\
&\qquad{}+\frac{1+q^{2m+1}}{1+q}\,(1-q)(2-q)+\frac{q-q^{2m+1}}{1+q}\,(1-q)\\
&=(2m+2)\,\frac{1-q}{1+q}+(1-q^{2m+2})\,\frac{q(1-q)}{(1+q)^2},
\end{align*}
proving the lemma when $n$ is even.

Now assume that \eqref{mean} holds for $n=2m-1$.  We must show that it holds for $n=2m+1$.  By the interpretation of the distribution (see Figure~\ref{odd-fig}),
\begin{align*}
E[Z_{2m+1}-Z_{2m-1}\mid Z_{2m-1}\text{ is even}]&=p=1-q,\\
E[Z_{2m+1}-Z_{2m-1}\mid Z_{2m-1}\text{ is odd}]&=pq+2p^2=(1-q)(2-q).
\end{align*}
We conclude from the induction hypothesis and Lemma~\ref{even-lemma} that
\begin{align*}
E[Z_{2m+1}]&=E[Z_{2m-1}]\\
&\qquad{}+P(Z_{2m-1}\text{ is even})E[Z_{2m+1}-Z_{2m-1}\mid Z_{2m-1}\text{ is even}]\\
&\qquad{}+P(Z_{2m-1}\text{ is odd})E[Z_{2m+1}-Z_{2m-1}\mid Z_{2m-1}\text{ is odd}]\\
&=(2m-1)\,\frac{1-q}{1+q}+(1+q^{2m-1})\,\frac{q(1-q)}{(1+q)^2}\\
&\qquad{}+\frac{q+q^{2m}}{1+q}\,(1-q)+\frac{1-q^{2m}}{1+q}\,(1-q)(2-q)\\
&=(2m+1)\,\frac{1-q}{1+q}+(1+q^{2m+1})\,\frac{q(1-q)}{(1+q)^2},
\end{align*}
proving the lemma when $n$ is odd.
\end{proof}

We conclude this section with alternative interpretations of the distribution of $Z_n$, given by \eqref{dist-even} if $n=2m$ and by \eqref{dist-odd} if $n=2m-1$, that do not require separate formulations for $n$ even and $n$ odd.  
\begin{itemize}

\item Consider a particle that starts at $(0,0)$.  At each time step, it moves one unit to the right with probability $p$ or one unit up with probability $q$, stopping at the first time it reaches or crosses the boundary $(k,(n-k)/2)$, $k=0,1,\ldots,n$.  Let $Z_n$ denote the $x$-coordinate of its final position.

\item Consider a particle that starts at $(0,0)$.  At each time step, it moves one unit to the right with probability $p$ or two units up with probability $q$, stopping at the first time it reaches or crosses the boundary $(k,n-k)$, $k=0,1,\ldots,n$.  Let $Z_n$ denote the $x$-coordinate of its final position.

\item Consider a particle that starts at $(0,0)$.  At each time step, it moves one unit to the right with probability $p$ or one unit up with probability $q$ followed by another unit up with probability 1, stopping at the first time it reaches the boundary $(k,n-k)$, $k=0,1,\ldots,n$.  Let $Z_n$ denote the $x$-coordinate of its final position.
\end{itemize}

The last of these interpretations is the context in which the distribution arises in Section~\ref{rate section} below.

\section{Proofs of Theorems \ref{sup=1} and \ref{rate}}\label{rate section}

\begin{proof}[Proof of Theorem \ref{sup=1}]
The result is immediate from Theorem~\ref{rate} provided we can show that $f(\rho):=\mu(r,\rho,(AB)^s B^{r-2})$ is continuous at 0.  We use Theorem~\ref{SLLN}, $3''$, or \ref{SLLN-2} to evaluate $f(\rho)$, which is a rational function of $\rho$.  The only potential singularities are those of the stationary distribution $\bm\pi$ (or $\bm\pi_1$ or $\bm\pi_2$).  But the existence and uniqueness of $\bm\pi$ (or $\bm\pi_1$ or $\bm\pi_2$) for $0\le\rho\le1$ ensures that $f(\rho)$ is real analytic there, hence continuous.
\end{proof}

We give two proofs of Theorem~\ref{rate}, the first one direct (depending solely on Theorems~$3'$ and \ref{SLLN-2}) but complicated, and the second one more easily understood but depending on Theorems~$3'$ and \ref{SLLN-2} and Lemmas~\ref{even-lemma} and \ref{mean-lemma}.

\begin{proof}[First proof of Theorem~\ref{rate}]
First assume that $r\ge3$ is odd.
Since $\dot{\bm P}_A \bm 1=\bm 0$, Theorem~$3'$ tells us that the rate of profit, regardless of initial capital, can be expressed as
\begin{align}\label{mean profit}
&\mu(r,0,(AB)^sB^{r-2})\nonumber\\
&\quad{}=(2s+r-2)^{-1} \bm \pi  \bigg[ \sum_{j=0}^{s-1}(\bm P_A \bm P_B)^j\bm P_A \dot{\bm P}_B +  \sum_{i=0}^{r-3}(\bm P_A \bm P_B)^s (\bm P_B)^i\dot{\bm P}_B \bigg]\bm 1,
\end{align}
where $\bm \pi$ is the stationary distribution of $\bm P:=(\bm P_A \bm P_B)^s ({\bm P}_B)^{r-2}$. Since $\rho=0$ and $r$ is odd, $\bm P$ is reducible with one recurrent class $\{0, r-1\}$ and $r-2$ transient states. From the observations about the pattern $(AB)^sB^{r-2}$ in Section~\ref{intro} it follows that $\bm P(0,0)=1-\bm P(0,r-1)=1-2^{-s}$ and $\bm P(r-1,0)=1$ so that the stationary distribution $\bm \pi$ is given by $\bm\pi=(\pi_0, 0, 0, \ldots, 0, \pi_{r-1})$, where 
\begin{align*}
\pi_0 = 1- \pi_{r-1}= \frac{2^s}{2^s+1}.
\end{align*}

Except for the factor $(2s+r-2)^{-1}$, all of the terms in \eqref{mean profit} have the form $\bm\pi\bm\Lambda\dot{\bm P}_B\bm1$ for a transition matrix $\bm \Lambda=(\lambda_{i,j})_{i,j=0,1,\ldots,r-1}$. Therefore, using $\dot{\bm P}_B\bm1=(-1,1,1,\ldots,1)^\textsf{T}$, we have
\begin{equation}\label{term}
\bm\pi\bm\Lambda\dot{\bm P}_B\bm1=1-2(\pi_0 \lambda_{0,0}+\pi_{r-1}\lambda_{r-1,0}), 
\end{equation}
showing that we need only determine two of the entries of $\bm\Lambda$ to evaluate \eqref{term}.

We first consider the transition matrix $\bm\Lambda=(\bm P_A \bm P_B)^j\bm P_A$ for $0\leq j\leq s-1$.
When $j<(r-1)/2$, we have $\lambda_{0,0}=0$. When $j\geq(r-1)/2$, from state 0 we can reach state $r-1$ after $j$ plays of $AB$ if there are at least $(r-1)/2$ wins from the $j$ plays of game $A$, after which we can move to state 0 with an additional win from game $A$.  Thus, we have
\begin{align*}
\lambda_{0,0}=\sum_{k=(r-1)/2}^j \binom{j}{k}\frac{1}{2^{j+1}}.
\end{align*}
For all $j$, we have $\lambda_{r-1,0}=1/2.$ Using \eqref{term}, for $j<(r-1)/2$,
\begin{align*}
\bm \pi (\bm P_A \bm P_B)^j\bm P_A \dot{\bm P}_B \bm 1=1-2\pi_{r-1}\,\frac12 = \pi_0 = \frac{2^s}{2^s+1},
\end{align*}
and for $j\geq(r-1)/2$, 
\begin{align*}
\bm \pi (\bm P_A \bm P_B)^j \bm P_A \dot{\bm P}_B \bm 1 &= 1-2\bigg[\pi_0 \sum_{k=(r-1)/2}^j \binom{j}{k}\frac{1}{2^{j+1}}  + \pi_{r-1}\,\frac12\bigg] \\
&= \frac{2^s}{2^s+1}\bigg[1-\sum_{k=(r-1)/2}^j \binom{j}{k}\frac{1}{2^j}\bigg].
\end{align*}
Summing these $s$ terms, we have
\begin{align}
\label{first}
\sum_{j=0}^{s-1}\bm \pi (\bm P_A \bm P_B)^j\bm P_A \dot{\bm P}_B \bm 1=
\frac{2^s}{2^s+1}\bigg[s - \sum_{j=(r-1)/2}^{s-1}\;\sum_{k=(r-1)/2}^j \binom{j}{k} \frac{1}{2^j}\bigg].
\end{align}

Next we consider the transition matrix $\bm\Lambda=(\bm P_A \bm P_B)^s(\bm P_B)^i$ for $0\leq i\leq r-3$.
For even $i$, we have $\lambda_{0,0}= 2^{-s}$ and $\lambda_{r-1,0}=0$, from which we obtain, via \eqref{term},
\begin{align*}
\bm \pi (\bm P_A \bm P_B)^s(\bm P_B)^i \dot{\bm P}_B \bm 1=1-2\pi_0\,2^{-s} = \frac{2^s -1}{2^s+1}.
\end{align*}
Now let $i$ be odd.  Assume we start from state 0. With at least $(r-i)/2$ wins from $s$ plays of game $A$, we can reach state $r-i$ or an even state to its right after $s$ plays of game $AB$, and then move to state 0 after $i$ additional plays of game $B$.  Thus, we have
\begin{align*}
\lambda_{0,0}=\sum_{k=(r-i)/2}^s \binom{s}{k}\frac{1}{2^s}.
\end{align*}
Moreover, $\lambda_{r-1, 0}= 1$.  Thus, for odd $i$ we obtain, via \eqref{term},
\begin{align*}
\bm \pi (\bm P_A \bm P_B)^s(\bm P_B)^i \dot{\bm P}_B \bm 1&=1-2\bigg[\pi_0 \sum_{k=(r-i)/2}^s \binom{s}{k} \frac{1}{2^s}+\pi_{r-1}\bigg] \\
&=\frac{2^s-1}{2^s+1}-\frac{2}{2^s+1}\sum_{k=(r-i)/2}^s \binom{s}{k}.
\end{align*}
Summing over $i$, we have
\begin{equation}\label{second}
\quad \sum_{i=0}^{r-3}\bm \pi (\bm P_A \bm P_B)^s(\bm P_B)^i \dot{\bm P}_B \bm 1=
(r-2)\frac{2^s-1}{2^s+1}-\frac{2}{2^s+1}\sum_{i=1}^{(r-3)/2}\!\!\!\sum_{k=(r-2i+1)/2}^s \binom{s}{k}. \quad
\end{equation}

For the double sum in \eqref{second}, a change of variables gives
\begin{equation*}
\sum_{i=1}^{(r-3)/2}\sum_{k=(r-2i+1)/2}^s \binom{s}{k}=\sum_{j=2}^{(r-1)/2}\sum_{k=j}^s \binom{s}{k}.
\end{equation*}
There are two cases.  If $(r-1)/2\ge s$, which also makes the double sum in \eqref{first} zero, then this becomes
\begin{align}\label{identity}
\sum_{j=2}^s\sum_{k=j}^s \binom{s}{k}&=\sum_{k=2}^s\sum_{j=2}^k\binom{s}{k}=\sum_{k=2}^s(k-1)\binom{s}{k}=\sum_{k=0}^s(k-1)\binom{s}{k}+1\nonumber\\
&=s2^{s-1}-2^s+1=\frac{s2^s-2(2^s-1)}{2},
\end{align}
and \eqref{mean profit} becomes
\begin{align*}
\mu(r,0,(AB)^sB^{r-2})&=\frac{1}{2s+r-2}\bigg[\frac{s2^s}{2^s+1}+(r-2)\frac{2^s-1}{2^s+1}-\frac{s2^s-2(2^s-1)}{2^s+1}\bigg]\nonumber\\
&=\frac{r}{2s+r-2}\,\frac{2^s-1}{2^s+1}.
\end{align*}

If $(r-1)/2<s$, it suffices to verify the following identity:
\begin{align*}
\sum_{j=(r-1)/2}^{s-1}\sum_{k=(r-1)/2}^j\binom{j}{k} 2^{s-j}+2\sum_{j=2}^{(r-1)/2}\sum_{k=j}^s\binom{s}{k}=s2^s-2(2^s-1).
\end{align*}
For $1\le s_0<s$, 
\begin{align*}
&\sum_{j=s_0}^{s-1}\sum_{k=s_0}^j\binom{j}{k}2^{s-j}+2\sum_{j=2}^{s_0}\sum_{k=j}^s\binom{s}{k}-[s2^s-2(2^s-1)]\\
&\qquad{}=\sum_{j=s_0}^{s-1}\sum_{k=s_0}^j\binom{j}{k}2^{s-j}-2\sum_{j=s_0+1}^{s}\sum_{k=j}^s\binom{s}{k}\\
&\qquad{}=\sum_{k=s_0}^{s-1}\sum_{j=k}^{s-1}\binom{j}{k}2^{s-j}-2\sum_{k=s_0+1}^s\sum_{j=k}^s\binom{s}{j}\\
&\qquad{}=\sum_{k=s_0+1}^{s} 2^{s+1}\bigg[\sum_{j=k-1}^{s-1}\binom{j}{k-1}\frac{1}{2^{j+1}}-\sum_{j=k}^s\binom{s}{j}\frac{1}{2^s}\bigg]\\
&\qquad{}=0,
\end{align*}
where the first equality uses \eqref{identity} and the last equality uses the relationship between the binomial and negative binomial distributions.  (The first sum within brackets is the probability that, in a sequence of independent Bernoulli trials with success probability $1/2$, at most $s$ trials are needed for the $k$th success, and the second sum is the probability that at least $k$ successes occur in $s$ trials.)

Next assume that $r\ge4$ is even.  Theorem~\ref{SLLN-2} tells us that the rate of profit can be expressed as
\begin{align}\label{mean profit 2}
&\mu(r,0,(AB)^sB^{r-2})\nonumber\\
&\quad{}=(2s+r-2)^{-1} \bm \pi_0\bigg[ \sum_{j=0}^{s-1}(\bm P_A \bm P_B)^j\bm P_A \dot{\bm P}_B +  \sum_{i=0}^{r-3}(\bm P_A \bm P_B)^s (\bm P_B)^i\dot{\bm P}_B \bigg]\bm 1,
\end{align}
where $\bm\pi_0:=(1,0,0,\ldots,0)$ if initial capital is even and $\bm\pi_0:=(0,0,\ldots,0,1)$ if initial capital is odd.

Except for the factor $(2s+r-2)^{-1}$, all of the terms in \eqref{mean profit 2} have the form $\bm\pi_0\bm\Lambda\dot{\bm P}_B\bm1$ for a transition matrix $\bm \Lambda=(\lambda_{i,j})_{i,j=0,1,\ldots,r-1}$, and
\begin{equation*}
\bm\pi_0\bm\Lambda\dot{\bm P}_B\bm1=\begin{cases}1-2\lambda_{0,0}&\text{if initial capital is even},\\1-2\lambda_{r-1,0}&\text{if initial capital is odd}.\end{cases}
\end{equation*} 

For $\bm\Lambda=(\bm P_A\bm P_B)^j\bm P_A$ with $0\le j\le s-1$, $\lambda_{0,0}=0$ and $\lambda_{r-1,0}=1/2$.  For $\bm\Lambda=(\bm P_A\bm P_B)^s(\bm P_B)^i$ with $0\le i\le r-3$, $\lambda_{0,0}=0$ if $i$ is odd and
$$
\lambda_{0,0}=\bigg[\binom{s}{0}+\sum_{m=1}^{\lceil 2s/r\rceil}\sum_{k=(mr-i)/2}^{mr/2}\binom{s}{k}\bigg]\frac{1}{2^s}
$$
if $i$ is even.  Finally, $\lambda_{r-1,0}=1$ if $i$ is odd and $\lambda_{r-1,0}=0$ if $i$ is even.

Therefore, if initial capital is odd,
$$
\mu(r,0,(AB)^s B^{r-2})=\frac{1}{2s+r-2}\bigg[s\bigg(1-2\cdot\frac12\bigg)+\frac{r-2}{2}\,(1-1)\bigg]=0,
$$
and if initial capital is even,
\begin{align*}
&\mu(r,0,(AB)^s B^{r-2})\\
&\qquad{}=\frac{1}{2s+r-2}\bigg\{s+r-2-2\sum_{i=0}^{r/2-2}\bigg[\binom{s}{0}+\sum_{m=1}^{\lceil 2s/r\rceil}\;\sum_{k=mr/2-i}^{mr/2}\binom{s}{k}\bigg]\frac{1}{2^s}\bigg\}.
\end{align*}
It remains to check that this last expression coincides with the formula in \eqref{formula2}.  The quantity within braces is equal to
\begin{align*}
&s+r-2-2\bigg(\frac{r}{2}-1\bigg)\frac{1}{2^s}-2\sum_{m=1}^{\lceil 2s/r\rceil}\sum_{j=(m-1)r/2+2}^{mr/2}\sum_{k=j}^{mr/2}\binom{s}{k}\frac{1}{2^s}\\
&\quad{}=s+(r-2)\bigg(1-\frac{1}{2^s}\bigg)-2\sum_{m=1}^{\lceil 2s/r\rceil}\sum_{k=(m-1)r/2+2}^{mr/2}(k-1-(m-1)r/2)\binom{s}{k}\frac{1}{2^s}\\
&\quad{}=s+(r-2)\bigg(1-\frac{1}{2^s}\bigg)-2\sum_{m=1}^{\lceil 2s/r\rceil}\sum_{k=(m-1)r/2+1}^{mr/2}(k-1-(m-1)r/2)\binom{s}{k}\frac{1}{2^s}\\
&\quad{}=s+(r-2)\bigg(1-\frac{1}{2^s}\bigg)-2\sum_{k=0}^s(k-1)\binom{s}{k}\frac{1}{2^s}-\frac{2}{2^s}\\
&\qquad\quad{}+r\sum_{m=1}^{\lceil 2s/r\rceil}(m-1)\sum_{k=(m-1)r/2+1}^{mr/2}\binom{s}{k}\frac{1}{2^s}\\
&\quad{}=s+(r-2)\bigg(1-\frac{1}{2^s}\bigg)-2\bigg(\frac{s}{2}-1\bigg)-\frac{2}{2^s}\\
&\qquad\quad{}-r\bigg(1-\frac{1}{2^s}\bigg)+r\sum_{m=1}^{\lceil 2s/r\rceil}m\sum_{k=(m-1)r/2+1}^{mr/2}\binom{s}{k}\frac{1}{2^s}\\
&\quad{}=r\sum_{m=1}^{\lceil 2s/r\rceil}m\sum_{k=(m-1)r/2+1}^{mr/2}\binom{s}{k}\frac{1}{2^s}\\
&\quad{}=r\,\sum_{k=0}^s\bigg\lceil\frac{2k}{r}\bigg\rceil\binom{s}{k}\frac{1}{2^s},
\end{align*}
and the proof is complete.
\end{proof}

\begin{proof}[Second proof of Theorem~\ref{rate}]
First, fix an odd integer $r\ge3$ and a positive integer $s$.
We apply Theorem~$3'$ assuming \eqref{Sigma}--\eqref{W} with $\rho=0$ in \eqref{rho-param} and $C_1C_2\cdots C_t=(AB)^s B^{r-2}$ with $t:=2s+r-2$, to conclude that
\begin{equation}\label{limit}
\mu(r,0,(AB)^s B^{r-2})=\lim_{n\to\infty}(nt)^{-1}E[S_{nt}].
\end{equation}
(The theorem tells us that the rate of profit does not depend on initial capital, so for convenience we take initial capital congruent to 0 (mod $r$).)  Here $S_1,S_2,\ldots$ is the player's sequence of cumulative profits.  We can evaluate $E[S_{nt}]$.

We denote by $p_n(k)$, $k=0,1,\ldots,n$, the probability mass function in \eqref{dist-even} if $n=2m$ and in \eqref{dist-odd} if $n=2m-1$.  We claim that
$$
P(S_{nt}=kr-\text{mod}(n-k,2))=p_n(k),\quad k=0,1,\ldots,n,
$$
with $p=1-2^{-s}$.  The result follows by using the third of the alternative interpretations of the distribution in \eqref{dist-even} and \eqref{dist-odd} at the end of Section~\ref{mean section}. 

We can now evaluate, with the help of Lemmas~\ref{even-lemma} and \ref{mean-lemma}, mean profit after $nt$ games:
\begin{align*}
E[S_{nt}]&=\sum_{k=0}^n (kr-\text{mod}(n-k,2))p_n(k)\\
&=rE[Z_n]-P(n-Z_n\text{ is odd})\\
&=r\bigg(n\,\frac{1-q}{1+q}+[1-(-1)^n q^n]\,\frac{q(1-q)}{(1+q)^2}\bigg)-\frac{q-(-1)^n q^{n+1}}{1+q}.
\end{align*}
We divide by $nt=n(2s+r-2)$ and let $n\to\infty$ to obtain
$$
\lim_{n\to\infty}(nt)^{-1}E[S_{nt}]=\frac{r}{2s+r-2}\,\frac{1-q}{1+q}=\frac{r}{2s+r-2}\,\frac{2^s-1}{2^s+1},
$$
so \eqref{formula1} follows from this and \eqref{limit}.

Second, fix an even integer $r\ge4$ and a positive integer $s$.  We apply Theorem~\ref{SLLN-2} assuming \eqref{Sigma}--\eqref{W} with $\rho=0$ in \eqref{rho-param} and $C_1C_2\cdots C_t=(AB)^s B^{r-2}$ with $t:=2s+r-2$, to conclude that \eqref{limit} holds.
(The theorem tells us that the rate of profit depends on initial capital only through its parity, so for convenience we take initial capital congruent to 0 (mod $r$) if initial capital is even, or congruent to $r-1$ (mod $r$) if odd.)  Recalling from Section \ref{intro} that, with initial capital congruent to 0 (mod $r$), each play of $(AB)^s B^{r-2}$ results in a mean profit of 
\begin{align*}
E[S_t]=\sum_{m=1}^{\lceil 2s/r\rceil}mr\sum_{k=(m-1)r/2+1}^{mr/2}\binom{s}{k}\frac{1}{2^s}=r\,\sum_{k=0}^s\bigg\lceil\frac{2k}{r}\bigg\rceil\binom{s}{k}\frac{1}{2^s},
\end{align*}
we find that
$$
\lim_{n\to\infty}(nt)^{-1}E[S_{nt}]=\frac{r}{2s+r-2}\,\sum_{k=0}^s\bigg\lceil\frac{2k}{r}\bigg\rceil\binom{s}{k}\frac{1}{2^s}.
$$
With initial capital congruent to $r-1$ (mod $r$), $P(S_{nt}=0)=1$, so 
$$
\lim_{n\to\infty}(nt)^{-1}E[S_{nt}]=0,
$$
and \eqref{formula2} follows from the last two limits and \eqref{limit}.
\end{proof}

\section*{Acknowledgments}

We are grateful to Derek Abbott for raising the question addressed here and to Ira Gessel for suggesting the lattice path interpretation of the distribution defined by \eqref{dist-even} and \eqref{dist-odd}.

\end{document}